\documentclass[11pt]{amsart}
\usepackage{graphicx}
\usepackage{amsmath}
\usepackage{amssymb}
\usepackage{amsfonts}
\usepackage{verbatim}
\usepackage{scalefnt}
\usepackage{multirow}
\usepackage{enumerate}
\usepackage{mathtools}
\usepackage{float}
\usepackage{enumitem}
\restylefloat{figure}
\usepackage[margin=3cm]{geometry}

\usepackage{fancyhdr}
\usepackage{hyperref}

\newcommand{\cC}{{\mathcal C}}
\newcommand{\cL}{{\mathcal L}}

\newtheorem{theorem}{Theorem}

\newtheorem{claim}[theorem]{Claim}

\newtheorem*{definition*}{Definition}

\numberwithin{equation}{section}
\numberwithin{theorem}{section}

\title[The number of maximum primitive sets of integers]%
  {The number of maximum primitive sets of integers}

\author{Hong Liu}
\email{h.liu9@warwick.ac.uk}
\address{Mathematics Institute, University of Warwick, Coventry, CV4 7AL, UK}

\author{P\'eter P\'al Pach}
\email{ppp@cs.bme.hu}
\address{Department of Computer Science and DIMAP, University of Warwick, Coventry CV4 7AL, UK and Department of Computer Science and Information Theory, Budapest
  University of Technology and Economics, 1117 Budapest, Magyar tud\'osok
  k\"or\'utja 2., Hungary}

\author{Rich\'ard Palincza}
\email{pricsi@cs.bme.hu}
\address{Department of Computer Science and Information Theory, Budapest
  University of Technology and Economics, 1117 Budapest, Magyar tud\'osok
  k\"or\'utja 2, Hungary}

\thanks{H.L.\ was supported by the Leverhulme Trust Early Career Fellowship~ECF-2016-523.
P.P.P. was partially supported by the National Research, Development and Innovation Office NKFIH (Grant Nr.~PD115978) and the J\'anos Bolyai Research Scholarship of the Hungarian Academy of Sciences; he has also received funding from the European Research Council (ERC) under the European Union’s Horizon 2020 research and innovation programme (grant agreement No 648509). This publication reflects only its author's view; the European Research Council Executive Agency is not responsible for any use that may be made of the information it contains.}

\begin{document}

\begin{abstract}	
	A set of integers is \emph{primitive} if it does not contain an element dividing another. Denote by $f(n)$ the number of maximum-size primitive subsets of $\{1,\ldots, 2n\}$. We prove that the limit $\alpha=\lim_{n\rightarrow \infty}f(n)^{1/n}$ exists. Furthermore, we present an algorithm approximating $\alpha$ with $(1+\varepsilon)$ multiplicative error in $N(\varepsilon)$ steps, showing in particular that $\alpha\approx 1.318$. Our algorithm can be adapted to estimate also the number of all primitive sets in $\{1,\ldots, n\}$.
	
	We address another related problem of Cameron and Erd\H{o}s. They showed that the number of sets containing pairwise coprime integers in $\{1,\ldots, n\}$ is between $2^{\pi(n)}\cdot e^{(\frac{1}{2}+o(1))\sqrt{n}}$ and $2^{\pi(n)}\cdot e^{(2+o(1))\sqrt{n}}$. We show that neither of these bounds is tight: there are in fact $2^{\pi(n)}\cdot e^{(1+o(1))\sqrt{n}}$ such sets.
\end{abstract}

\date{\today}
\maketitle

\section{Introduction}
Enumeration problem is one of the central topics in combinatorics. In the past decade, it has attracted a great deal of attention. We refer the readers to~\cite{BMS,ST} for the literature on enumeration problems on graphs and other settings; see also~\cite{BLSh17,Hancock-Staden-Treglown,Tran} for more recent results in arithmetic setting forbidding additive structures.

In this paper, we consider enumeration problems of different flavours. We are interested in sets of integers with multiplicative forbidden structure. We say that a set of integers is \emph{primitive} if it does not contain any element dividing another. Let $g(n)$ be the number of primitive subsets of $[n]:=\{1,\ldots,n\}$. Cameron and Erd\H{o}s \cite{CamErd} proved that for sufficiently large $n$, $1.55967^n\leq g(n)\leq 1.6^n$ and conjectured that the limit of $g(n)^{1/n}$ exists. Recently, Angelo \cite{Angelo} verified this conjecture. However he was unable to provide a method to find better estimate of the limit.

A natural extension is to estimate the cardinality of certain subfamily consisting of ``largest'' members. For example, in the context of graphs, the number of maximal (under graph inclusion) triangle-free graphs on vertex set $[n]$ has been studied~\cite{BLPSh}; and in the arithmetic setting, the number of maximal (under set inclusion) sum-free subsets of $[n]$ has been considered~\cite{BLShT15,BLShT18+}. One such question about primitive sets, that is, the number of maximum primitive subsets of $[2n]$, was mentioned by Bishnoi in his blog post ``On a famous pigeonhole problem''~\cite{Bishnoi}. Note that a maximum primitive subset of $[2n]$ is of size $n$. Indeed, group elements of $[2n]$ into $n$ classes according to their largest odd divisor, then a primitive set can have at most one element from each class. On the other hand $\{n+1,\ldots,2n\}$ is primitive. Behind this pigeonhole argument was the famous story that young P\'osa came up with this solution during a dinner with Erd\H{o}s. Coming back to the question mentioned by Bishnoi, denoting by $f(n)$ the number of $n$-element primitive subsets of $[2n]$, Vijay \cite{Vijay} proved that for sufficiently large $n$, $1.303^n\leq f(n)\leq 1.408^n$. Interestingly, the sequence $f(n)$ was already considered in OEIS, the Online Encyclopedia of Integer Sequences~\cite{OEIS}. A priori, it was not clear whether the limit of $f(n)^{1/n}$ exists. Our main result answers this question, showing that this is indeed the case, i.e.~$f(n)^{1/n}$ converges to some $\alpha$, which is roughly $1.318$.
\begin{theorem}\label{thm-max-primitive}
	The number of $n$-element primitive subsets of $[2n]$ is $f(n)=(\alpha+o(1))^n$. Futhermore, for any $\varepsilon>0$, there exists $N(\varepsilon)$ such that $\alpha$ can be approximated with a multiplicative error $1+\varepsilon$ in $N(\varepsilon)$ steps.
\end{theorem}
The function $N(\varepsilon)$ here can be explicitly given. Our algorithm provides upper ($\approx 1.3184$) and lower ($\approx 1.3183$) estimates within a ratio of $1.0001$. It can be adapted to estimate the limit of $g(n)^{1/n}$ as well, providing  upper ($\approx 1.5745$) and lower ($\approx 1.571$) estimates within a ratio of $1.0022$. 

\medskip

\noindent\emph{Remark 1:}
Our proof for convergence of $f(n)^{1/n}$ starts in a similar way as~\cite{Angelo} by looking at the divisibility lattices, however, the problem for $f(n)$ turns out to be more difficult partly due to the fact that the maximum property is not monotone. In other words, the number of max-size primitive subsets in a subset $X$ of $A$ is not necessarily a lower bound for the number of max-size primitive subsets in $A$,~e.g.~$f(7)=12>f(8)=10$~\cite{OEIS}. Several new ideas are needed to overcome this issue.

\medskip

\noindent\emph{Remark 2:}
To obtain estimates on the limit, our algorithm (Section~\ref{sec-numeric-max-primitive}) takes one step beyond the proof of convergence to look into the structure of the divisibility lattices of different sizes and analyse how they evolve over time as we merge them. We would also like to point out that although proving the convergence of $g(n)^{1/n}$ is simpler than that of $f(n)^{1/n}$ from the theoretical point of view, better approximation ratio is obtained for $f(n)$ under the same computational time because the parameters involved are much smaller.

\medskip

Another related problem concerns subsets $S$ of $[n]$ containing pairwise coprime elements, i.e.~$\forall~a,b$ in $S$, $(a,b)=1$. Cameron and Erd\H{o}s~\cite{CamErd} showed that the number of such sets is between $2^{\pi(n)}\cdot e^{(\frac{1}{2}+o(1))\sqrt{n}}$ and $2^{\pi(n)}\cdot e^{(2+o(1))\sqrt{n}}$. Deciding which bound is closer to the truth remains an interesting question. We resolve this question, showing that neither of these bounds is tight.

\begin{theorem}\label{thm-coprime}
	The number of subsets of $[n]$ containing pairwise coprime integers is $2^{\pi(n)}\cdot e^{(1+o(1))\sqrt{n}}$.
\end{theorem}

The rest of the paper is organised as follows. We will present the proof of  Theorem~\ref{thm-max-primitive} in Section~\ref{sec-max-primitive}, and the algorithm approximating $\alpha$ in Section~\ref{sec-numeric-max-primitive}. We highlight how to adapt the algorithm for $g(n)$ in Section~\ref{sec-primitive}. The proof of~\ref{thm-coprime} is in Section~\ref{sec-coprime}.

\section{Number of maximum primitive subsets of $[2n]$}\label{sec-max-primitive}
Let $l$ be a positive integer, $P_l=\{p_1,p_2,\dots,p_l\}$ be the set of first $l$ primes, and $q:=p_{l+1}$ be the $(l+1)$st prime. Let $M_l\subseteq \mathbb N$ denote the set of all positive integers whose prime factors are all in $P_l$:
$$M_l=\{p_1^{\alpha_1}p_2^{\alpha_2}\dots p_l^{\alpha_l}: \alpha_1,\dots,\alpha_l\geq 0\}.$$
For $x>0$ let $M_l(x)=\{n\in M_l: n\leq x\}$. Note that $|M_l(x)|\le (1+\log_2x)^l$ as the number of choices for each $\alpha_i$ is at most $1+\log_2x$.


Partition $[2n]$ into classes in such a way that two elements belong to the same class if and only if their ratio can be written as a ratio of two elements from $M_l$. In other words, denoting $T_l=\{t\in[2n]: (t,p_1p_2\ldots p_l)=1\}$, the lattice containing $t\in T_l$ is the following:
$$L_l(t)=t\cdot M_l(2n/t)=\{tp_1^{\alpha_1}p_2^{\alpha_2}\dots p_l^{\alpha_l}:\alpha_1,\dots,\alpha_l\geq 0, tp_1^{\alpha_1}p_2^{\alpha_2}\dots p_l^{\alpha_l}\leq 2n \}.$$
Let $K$ be a positive integer. Define $T_{l,K}=T_l\cap (\frac{2n}{K+1}, \frac{2n}{K}]$, $T_{l,\le K}=\cup_{i\in[K]}T_{l,i}$, and $\cL_{l,K}=\cup_{t\in T_{l,\le K}} L_l(t)$. For brevity, throughout this section, we skip the subscript $l$ and write instead $M,T,L(t),T_K,T_{\le K}$ and $\cL_K$ whenever it is clear from the context.

Equip $L(t)$ with the divisibility partial ordering. Note that in $L(t)$, the size of a maximum antichain is the number of odd elements in $L(t)$. Indeed, by pigeonhole principle the size of an antichain cannot be larger; on the other hand, each odd element of $L(t)$ has a unique multiple in $L(t)\cap(n,2n]$, all of which form an antichain.


The following fact about max-size primitive sets will be frequently used. For each odd $u$ in $[2n]$, denote by $C(u)$ the chain $\{u,2u,4u,\ldots\}$.
\begin{itemize}
	\item[$(\dagger)$] If $A\subseteq [2n]$ is an $n$-element primitive subset, then $A$ contains exactly one element from each chain $C(u)$.
\end{itemize}

Note that if $A$ is a max-size primitive set, then $A\cap L(t)$ is a max-size antichain for all $t\in T$. However, the converse is not true as the partial order in $L(t)$ does not forbid prime divisors at least $q$. The following claim establishes a partial converse.

\begin{claim}\label{cl-large-primitive}
	Let $T^*\subseteq T$. If $A^*\subseteq (2n/q,2n]$ is a union of max-size antichains of $L(t)$, $t\in T^*$, then $A^*$ can be extended to an $n$-element primitive set $A^*\subseteq A\subseteq (2n/q,2n]$. Futhermore, such extension is injective, i.e.~distinct $A^*$ extend to distinct $A$.
\end{claim}
\begin{proof}
 Note that each chain $C(u)$ is either contained in $L(t)$ for some $t\in T^*$ or disjoint from $\cL^*:=\cup_{t\in T^*}L(t)$. By~$(\dagger)$, to extend $A^*$, we need to pick exactly one element from each chain disjoint from $\cL^*$. 
 
 The elements of $A^*$ do not divide any element in $(2n/q,2n]\setminus \cL^*$. Indeed, for any $a\in A^*$ and $b\notin \cL^*$, if $b/a<q$ is an integer, then $b$ is in the same lattice as $a$, a contradiction. Thus, we only have to guarantee that the newly chosen elements form an antichain in $(2n/q,2n]\setminus \cL^*$ and none of them divides any element of $A^*$. This holds, if from each chain  $C(u)$ disjoint from $\cL^*$, we choose the unique element lying in $(n,2n]$. Clearly, the resulting $n$-element primitive set $A$ satisfies $A^*\subseteq A\subseteq (2n/q,2n]$. As $A^*=A\cap \cL^*$, this extension is injective.
\end{proof}

Let $f_q(n)$ be the number of $n$-element primitive sets in $(2n/q,2n]$. Clearly $f_q(n)\le f(n)$. We now give bounds on $f_q(n)$, which will be useful for proving the convergence of $f(n)^{1/n}$.
Note that if $\frac{2n}{i+1}<t\leq \frac{2n}{i}$, then the poset $L(t)$ is isomorphic with $M(i)$.  Consequently, an antichain in $L(t)\cap (2n/q,2n]$ corresponds to an antichain in $M(i)\cap (i/q,i]$, since $\frac{i}{q}\leq \frac{2n}{qt}<\frac{i+1}{q}$. Let $r'(i)=r_l'(i)$ denote the number of those max-size antichains of $M(i)$ in $(i/q,i]$. Then Claim~\ref{cl-large-primitive} implies that
$f_q(n)\ge \prod \limits_{i=1}^{K} r'(i)^{|T_i|}.$
Note that $|T_i|\approx\frac{2n}{i(i+1)}\prod\limits_{j=1}^l\left(1-\frac{1}{p_j}\right)$. Define $c'_{K}=c'_{l,K}=\prod\limits_{i=1}^K r'(i)^{\frac{2}{i(i+1)}\prod\limits_{j=1}^l\left(1-\frac{1}{p_j}\right)}$. Then the above inequality implies that for every $K$,
\begin{equation}\label{eq-lower}
	f(n)\ge f_q(n)\ge (c'_{K}+o(1))^n.
\end{equation}

Let $V_{K}:=(2n/q,2n]\setminus \cL_K$ be the left-over elements. Then $|V_{K}|=(\eta'_{K}+o(1))n$, where 
$\eta'_{K}=\eta'_{l,K}=\left(2-\frac{2}{q}\right)-\sum\limits_{i=1}^K \frac{2}{i(i+1)} \prod\limits_{j=1}^l\left(1-\frac{1}{p_j}\right)\left(|M(i)|-\left|M\left(\frac{i}{q}\right)\right|\right)$. Since $(2n/q,2n]$ are covered by $\cup_{i\in[2n], t\in T_i}L(t)$ and $|M(i)|\le (1+\log_2i)^l$, we have 
\begin{equation}\label{eq-eta}
	\eta'_K\leq \sum\limits_{i=K+1}^\infty \frac{2}{i(i+1)}\cdot (1+\log_2 i)^l\rightarrow 0,\quad \mbox{ as } \quad K\rightarrow \infty.
\end{equation}
Recall that each max-size primitive set must intersect each lattice at a max-size antichain, so 
$$f_q(n)\le 2^{|V_K|}\cdot\prod \limits_{i=1}^{K} r'(i)^{|T_i|} =(c'_{K}\cdot 2^{\eta'_{K}}+o(1))^n.$$

Let $s=\lfloor \log_2 q \rfloor$. We shall see that
\begin{equation}\label{eq-upper}
f(n)\le f_q(n)\cdot 2^{\frac{s+2}{2^s}\cdot n}\le (c'_{K}\cdot 2^{\eta'_{K}}\cdot 2^{\frac{s+2}{2^s}}+o(1))^n.
\end{equation}
Thus, combining~\eqref{eq-lower} and~\eqref{eq-upper}, we see that $f_q(n)^{1/n}$ and $c'_{K}$ converge to the same limit $\alpha_l$ as $K,n\rightarrow\infty$. From the definition of $f_q$ it follows that the sequence $\alpha_l$ is increasing and bounded. Hence, again by~\eqref{eq-lower} and~\eqref{eq-upper} we obtain that $f(n)^{1/n}$, $\alpha_l$ converge to the same limit $\alpha$ as $l,n\rightarrow\infty$.  

Let $\cC$ be the union of the chains contained in $(2n/q,2n]$, i.e.~$\cC=\cup_{u>\frac{2n}{q}}C(u)$. Let $A$ be a max-size primitive subset of $[2n]$. By~$(\dagger)$, $A':=A\cap \cC\subseteq (2n/q,2n]$ consists of exactly one element from each chain $C(u)$. Furthermore, $A'$ can be extended to an $n$-element primitive subset in $(2n/q,2n]$ by adding $C(u)\cap (n,2n]$ for each chain $C(u)$ not in $\cC$. Such extension is again injective as $A'=A\cap \cC$.  Therefore, the number of possibilities for $A'$ is at most $f_q(n)$. To bound the number of possibilities to extend $A'$ to $A$, the crude upper bound $2^{|[2n]\setminus \cC|}$ suffices. 

To estimate the size of $[2n]\setminus \cC$, recall the definition of $s$ implies that $s$ is the unique integer satisfying $2^s\leq q<2^{s+1}$. Then those chains $C(u)$ for which $2n/q<2n/2^s\leq u\leq 2n$ are contained in $\cC$. There are $n/2^s$ odd numbers in $[2n/2^s]$ and each corresponding chain has at most $s$ elements in $[2n/2^s,2n]$. So
$$|[2n]\setminus \cC|\le \frac{2n}{2^s}+\frac{n}{2^s}\cdot s= \left(\frac{s+2}{2^s}\right)n,$$
implying that
$$f(n)\leq f_q(n)\cdot 2^{|[2n]\setminus \cC|}\le (c'_{K}\cdot 2^{\eta'_{K}}\cdot 2^{\frac{s+2}{2^s}}+o(1))^n.$$
Setting $l= 10\log(1/\varepsilon)/\varepsilon$ and $K=K(l)=l^{10l\log\log l}$, we get, from~\eqref{eq-eta} and~\eqref{eq-upper}, that the limit satisfies $c'_{K}\leq \alpha \leq c'_{K} \cdot 2^{l^{-5l\log\log l}} \cdot 2^{\frac{s+2}{2^s}}\le (1+\varepsilon)c'_K$. In other words, to determine $\alpha$ up to a multiplicative error less than $1+\varepsilon$, it suffices to calculate $c'_{K(l)}$ for $l\approx \log(1/\varepsilon)/\varepsilon$.


Note that some of the estimations that we used are rough, if we calculate more precisely, much better estimations can be obtained (with a fixed $l$). For instance, for $l=2$, $K=10^6$ we get 
\begin{equation}\label{eq-lower-truncated-antichain}
	1.31464 \leq c'_{K}\leq \alpha.
\end{equation}

Finally, we mention another way of getting upper bounds for $\alpha$. Let $r_l(i)$ denote the total number of max-size antichains of $M_l(i)$ and let
$$c_{K}=c_{l,K}=\prod\limits_{i=1}^K r(i)^{\frac{2}{i(i+1)}\prod\limits_{i=1}^l\left(1-\frac{1}{p_i}\right)}$$
Then we have 
\begin{equation}\label{eq-upper-all-antichain}
	\alpha \leq c_{K}\cdot 2^{\eta_{K}},
\end{equation}
since the number of left-over elements is $(\eta_{K}+o(1))n$, where
$$\eta_{K}=2-\sum\limits_{i=1}^K \frac{2}{i(i+1)}\prod\limits_{j=1}^l(1-1/p_j) |M(i)|,$$
For $l=2,K=10^6$ we get from~\eqref{eq-upper-all-antichain} $\alpha\leq 1.32157$. Together with~\eqref{eq-lower-truncated-antichain}, we get lower and upper estimates within a ratio of
$\frac{1.32157}{1.31464}<1.0053$.

\section{Numerical bounds for max-size primitive sets}\label{sec-numeric-max-primitive}
In this section, we will present an improved algorithm to get better estimates than the ratio $1.0053$ obtained from~\eqref{eq-lower-truncated-antichain} and~\eqref{eq-upper-all-antichain}.


We have showed that $c'_{K(l)}$ converges to $\alpha$ from below, and it suffices to calculate $c'_{K(l)}$ for large enough $K$ and $l$ to get arbitrarily good lower bound for $\alpha$ using~\eqref{eq-lower-truncated-antichain}. However, for larger values of $l$ calculating the $r'(i)$ values is getting more and more difficult. Furthermore, the speed of convergence of $c'_{K(l)}$ to $\alpha$ is getting slower, that is, for larger $l$ we need to calculate $c'_{K}$ for a larger $K=K(l)$ to get better lower bounds. We get into the same difficulty when looking for upper bounds, since calculating the $r(i)$ values is also getting more difficult as $l$ increases. 
Note that for a fixed value of $K$ it can happen that the $c_{K}\cdot 2^{\eta_{K}}$ product is increasing as $l$ increases, as the number of left-over elements, $(\eta_{K}+o(1))n$, is larger for fixed $K$ and larger $l$. Because of this, within the same running time the resulting bound using $l$ primes can get worse and worse, as $l$ increases.

We shall combine the bounds obtained with different choices for $l$ (the number of primes) in a way, which will enable us to get significantly better lower and upper bounds within the same running time. 

Let $\cC_0,\cC_1,\ldots,\cC_k$ be a partition of all the chains $C(u)$ (hence it is also a partition of $[2n]$). The role of $\cC_0$ will be slightly different than the roles of other $\cC_i$. We will know the structure of $\cC_i$ for $i>0$ and we will think of $\cC_0$ as the set of left-over elements, about which we do not have any information except the size of $\cC_0$.

If $A\subseteq [2n]$ is a max-size primitive set, then for every $i$ the set $A\cap \cC_i$ must also be a max-size primitive set (max-size antichain) in $\cC_i$, due to $(\dagger)$. Denoting by $f(X)$ the number of max-size primitive subsets of $X$, we obtain that 
\begin{equation}\label{f_upp}
f(n)=f([2n])\leq \prod\limits_{i=0}^k f(\cC_i)\leq 2^{|\cC_0|}\prod\limits_{i=1}^k f(\cC_i).
\end{equation}
If we change the partition by merging $\cC_i$ and $\cC_j$ for some $0<i<j$, then, compared to the previous bound on the right hand side of \eqref{f_upp}, we gain a factor of $\frac{f(\cC_i\cup \cC_j)}{f(\cC_i)f(\cC_j)}\leq 1$.

Conversely, if $A_i\subseteq \cC_i$ is a max-size antichain in $\cC_i$ for every $i>0$, furthermore, all multiples (in $[2n]$) of all elements of $A_i$ are contained in $\cC_i$, then $\bigcup\limits_{i=1}^k A_i$, together with the unique element in $C(u)\cap (n,2n]$ from each chain $C(u)\in \cC_0$, forms a max-size primitive subset of $[2n]$. 
Therefore, denoting by $f^*(X)$ the number of those max-size primitive $A$ subsets of $X$ such that all multiples (in $[2n]$) of all elements of $A$ are contained in $X$, we obtain that 
\begin{equation}\label{f_low}
\prod\limits_{i=1}^k f^*(\cC_i)\leq \prod\limits_{i=0}^k f^*(\cC_i)\leq f([2n])=f(n).
\end{equation}

\subsection{Numerical estimations for the lower bound}

Let $S,K_1,\dots,K_S$ be positive integers, for brevity let $\underline{K}=(K_1,K_2,\dots,K_S)$.

\medskip

\noindent{\bf Step $0$.}
Let us consider first the $S$-dimensional lattices $\cL_{S,K_S}$, i.e.~$L_S(t)$ with $t\in T_{S,\le K_S}$. These are pairwise disjoint and as we have seen earlier, if we choose max-size antichains consisting of numbers larger than $2n/(tp_{S+1})$ in $M_S(2n/t)$, then their union, $A_S$, is a max-size antichain in $\cL_{S,K_S}$. Moreover, all multiples of all elements of $A_S$ are contained in $\cL_{S,K_S}$.

The number of choices for $A_S$ is $(\lambda_S(\underline{K})+o(1))^n$, where
$$\lambda_S(\underline{K}):=c'_{S,K_S}=\prod\limits_{i_S=1}^{K_S}\left(r_S'(i_S)\right)^{\frac{2}{i_S(i_S+1)}\prod\limits_{j\leq S}\left(1-\frac{1}{p_{j}}\right)}.$$


\medskip

\noindent
{\bf Step $1$.}
Now, consider the set of left-over elements $W_{S}:=[2n]\setminus \cL_{S,K_S}.$
Our aim is to add those lattices $L_{S-1}(t)$ for $t\in T_{S-1,\le K_{S-1}}$ that are contained in $W_S$. Let $t=p_S^{\alpha_S}t'$ with some $\alpha_S$ and $(t',p_1\dots p_S)=1$. The lattice $L_{S-1}(t)$ is contained in the left-over set $W_S$ if and only if $2n/(K_{S-1}+1)<t$ and $t'\leq 2n/(K_S+1)$. 
Hence, the number of those $2n/(i_{S-1}+1)<t\leq 2n/(i_{S-1})$ for which $L_{S-1}(t)\subseteq W_{S}$ is at least $(w(S-1,i_{S-1})+o(1))n$, where the weight $w(S-1,i_{S-1})$ can be calculated as follows: 
\begin{multline*}
w(S-1,i_{S-1}):=\frac{2}{i_{S-1}(i_{S-1}+1)}\prod\limits_{j\leq S-1}\left(1-\frac{1}{p_{j}}\right)\sum\limits_{\alpha_S=0}^{\infty}\left(1-\frac{1}{p_S}\right)\left(\frac{1}{p_S}\right)^{\alpha_S}I(p_S^{\alpha_S}\cdot i_{S-1}>K_S)=\\
= \frac{2\prod\limits_{j\leq S}\left(1-\frac{1}{p_{j}}\right)}{i_{S-1}(i_{S-1}+1)}\sum\limits_{\alpha_S=0}^{\infty}\frac{I(p_S^{\alpha_S}\cdot i_{S-1}>K_S)}{p_S^{\alpha_S}},
\end{multline*}
where $I(p_S^{\alpha_S}\cdot i_{S-1}>K_S)=1$, if $p_S^{\alpha_S}\cdot i_{S-1}>K_S$ holds, and 0 otherwise.

Therefore, the improvement compared to $\lambda_S(\underline{K})$ is the following factor:
$$\lambda_{S-1}(\underline{K}):=\prod\limits_{i_{S-1}=1}^{K_{S-1}}r_{S-1}'(i_{S-1})^{w(S-1,i_{S-1})}.$$

\medskip

\noindent{\bf Step $2$.}
Continuing like this, taking those $L_{S-2}(t)$ lattices in the set of left-over elements $W_{S-1}:=W_S\setminus \cL_{S-1,K_{S-1}}$ for which $t\in T_{S-2,\le K_{S-2}}$ the next improvement is the factor
$$\lambda_{S-2}(\underline{K}):=\prod\limits_{i_{S-2}=1}^{K_{S-2}}r_{S-2}'(i_{S-2})^{w(S-2,i_{S-2})},$$
where 
\begin{multline*}
w(S-2,i_{S-2}):=
\frac{2\prod\limits_{j\leq S}\left(1-\frac{1}{p_{j}}\right)}{i_{S-2}(i_{S-2}+1)}\sum\limits_{\alpha_S=0}^{\infty}\sum\limits_{\alpha_{S-1}=0}^{\infty}\frac{I(p_{S-1}^{\alpha_{S-1}}p_S^{\alpha_S}\cdot i_{S-2}>K_S)I(p_{S-1}^{\alpha_{S-1}}\cdot i_{S-2}>K_{S-1})}{p_{S-1}^{\alpha_{S-1}}p_{S}^{\alpha_{S}}}.
\end{multline*}

We continue this process, for $S-3\geq l\geq 1$. Step $(S-l)$ is as follows.

\medskip

\noindent{\bf Step $(S-l).$} 
Taking those $L_{l}(t)$ lattices contained in the set of left-over elements $W_{l+1}:=W_{l+2}\setminus \cL_{l+1,K_{l+1}}$ for which $t\in T_{l,\le K_{l}}$, the next improvement is
$$\lambda_{l}(\underline{K}):=\prod\limits_{i_{l}=1}^{K_{l}}r_{l}'(i_{l})^{w(l,i_{l})},$$
where the weight $w(l,i_{l})$ is

$$w(l,i_l):=
\frac{2\prod\limits_{j\leq S}\left(1-\frac{1}{p_{j}}\right)}{i_l(i_l+1)}\sum\limits_{\alpha_S=0}^{\infty}\sum\limits_{\alpha_{S-1}=0}^{\infty}\dots \sum\limits_{\alpha_{l+1}=0}^{\infty} \prod\limits_{v=l+1}^{S} \frac{I(i_lp_{l+1}^{\alpha_{l+1}}\dots p_v^{\alpha_v}>K_v)}{p_v^{\alpha_v}}.
$$

Finally, after Step $(S-1)$, we obtain the lower bound 
\begin{equation}\label{eq-f-lower-improved}
	\prod\limits_{l=1}^{S}\lambda_l(\underline{K})\leq \alpha.
\end{equation}
Note that assuming $K_1\geq K_2\geq \dots \geq K_S$ the formula for the weight $w(l,i_l)$ simplifies as
\begin{multline*}
w(l,i_l)=
\frac{2\prod\limits_{j\leq S}\left(1-\frac{1}{p_{j}}\right)}{i_l(i_l+1)}\sum\limits_{\alpha_S=0}^{\infty}\sum\limits_{\alpha_{S-1}=0}^{\infty}\dots \sum\limits_{\alpha_{l+1}=0}^{\infty}  \frac{I(i_lp_{l+1}^{\alpha_{l+1}}>K_{l+1})}{p_{l+1}^{\alpha_{l+1}}\dots p_S^{\alpha_S}}=\\
=\frac{2\prod\limits_{j\leq l+1}\left(1-\frac{1}{p_{j}}\right)}{i_l(i_l+1)} \sum\limits_{\alpha_{l+1}=0}^{\infty}  \frac{I(i_lp_{l+1}^{\alpha_{l+1}}>K_{l+1})}{p_{l+1}^{\alpha_{l+1}}} .
\end{multline*}
This assumption is natural, since calculating the $r'_l(i)$ values is getting more difficult as $l$ increases, meaning that with less primes we can calculate $r'_l(i)$ till larger values of $i$.


By taking $S=5,K_1=K_2=1006632960, K_3=50000, K_4=2695, K_5=1000$ we get that from~\eqref{eq-f-lower-improved} an improved lower bound
\begin{equation}\label{max_lower}
1.3183 \leq \alpha.
\end{equation}

\subsection{Numerical estimations for the upper bound}
We start with the partition $\cC_0=[2n]$, by \eqref{f_upp} this yields the trivial upper bound $f(n)\leq 4^n$, that is, $\alpha\leq 4$.

\medskip

\noindent{\bf Step $1$.} Fix a positive integer $K_1$. For every odd $t\in (2n/(K_1+1),n]$ we replace the one-element sets $t,2t,4t,\dots,2^{\lfloor\log_2 (2n/t)\rfloor}t$ by their union. (Note that performing this for every odd $t\in[2n]$ would result in the partition of $[2n]$ into chains $C(u)$.) After this, in the resulting partition every odd $t\in (2n/(K_1+1),2n]$ is contained in its chain $C(u)=L_1(t)$ and $\cC_0$ contains the left-over elements, specially, all positive integers up to $2n/(K_1+1)$. In every chain $C(u)$ the max-size antichain is a 1-element set, thus the number of max-size primitive sets is exactly the length of the chain $C(u)$.

Let $2\leq i_1\leq K_1$. The number of those odd $t$ for which $2n/(i_1+1)<t\leq 2n/i_1$  is $\frac{(1+o(1))n}{i_1(i_1+1)}$, and the improvement for these values of $t$ is $\frac{r_1(i_1)}{2^{|M_1(i_1)|}}=\frac{|M_1(i_1)|}{2^{|M_1(i_1)|}}$. Hence, the upper bound corresponding to the resulting partition is
$$\alpha\leq 4\prod\limits_{i_1=2}^{K_1} \left(\frac{|M_1(i_1)}{2^{|M_1(i_1)|}}\right)^{\frac{1}{i_1(i_1+1)}}.$$
Note that from now on the set of left-over elements, $\cC_0=[2n/(K_1+1)]$, will remain the same.

\medskip

\noindent{\bf Step $2$.} Let $K_2\leq K_1$ be a positive integer. For every odd $t\in (2n/(K_2+1),2n/3]$, in decreasing order, we replace the poset of $t$ and the poset of $3t$ by their union (for $n/9<t$ we have $L_2(3t)=L_{1}(3t)$). Note that when we get to $t$, the poset of $t$ is $C(t)=L_1(t)$ and the poset of $3t$ is $L_2(3t)$. Indeed, for $n/9<t$ we have $L_2(3t)=L_{1}(3t)$; while for $t\le n/9$, when we try to merge posets of $t$ and $3t$, we have already merged the posets of $3t$ and $9t$, which is $L_2(3t)$.

Let $3\leq i_2\leq K_2$. The number of those odd $t$ for which $2n/(i_2+1)<t\leq 2n/i_2$  is  $\frac{(1+o(1))n}{i_2(i_2+1)}$, and the improvement for these values of $t$ is the factor $\frac{r_2(i_2)}{r_1(i_2)r_2(i_2/3)}$, since the poset of $t$ is isomorphic with $M_1(i_2)$, the poset of $3t$ is isomorphic with $M_2(i_2/3)$ and the resulting poset is isomorphic with $M_2(i_2)$. Hence, the upper bound corresponding to the resulting partition is
$$\alpha\leq 4\prod\limits_{i_1=2}^{K_1} \left(\frac{|M_1(i_1)|}{2^{|M_1(i_1)|}}\right)^{\frac{1}{i_1(i_1+1)}}\prod\limits_{i_2=3}^{K_2} \left(\frac{r_2(i_2)}{r_1(i_2)r_2(i_2/3)}\right)^{\frac{1}{i_2(i_2+1)}}.$$


From now on, we continue as in Step 2.: with the help of the next prime $p_S$ we build  $S$-dimensional lattices from the existing $(S-1)$-dimensional lattices and calculate the improvement. For $3\leq S$ the general step is as follows:

\medskip

\noindent{\bf Step $S$.} Let $K_S\leq K_{S-1}$ be a positive integer. For every $t\in (2n/(K_S+1),2n/p_S]$ satisfying $(t,p_1p_2\dots p_{S-1})=1$, in decreasing order, we replace the poset of $t$ and the poset of $p_St$ by their union. When we consider $t$, the poset of $t$ is $L_{S-1}(t)$ and the poset of $p_St$ is $L_S(p_St,2n)$. (Note that for $2n/p_S^2<t$ we have $L_S(p_St,2n)=L_{S-1}(p_St,2n)$.)

Let $p_S\leq i_S\leq K_S$. The number of those $t$ for which $2n/(i_S+1)<t\leq 2n/i_S$ and $(t,p_1p_2\dots p_{S-1})=1$  is $\frac{(1+o(1))n}{i_S(i_S+1)}\prod\limits_{j<S}\left(1-\frac{1}{p_j}\right)$, and the improvement for these values of $t$ is the factor $\frac{r_S(i_S)}{r_{S-1}(i_S)r_S(i_S/p_S)}$, since the posets of $t$, $p_St$ and the resulting poset are isomorphic with $M_{S-1}(i_S)$, $M_S(i_S/p_S)$ and $M_S(i_S)$ respectively.
Hence, the upper bound corresponding to the resulting partition is
$$\alpha\leq 4\prod\limits_{i_1=2}^{K_1} \left(\frac{|M_1(i_1)|}{2^{|M_1(i_1)|}}\right)^{\frac{1}{i_1(i_1+1)}}\prod\limits_{l=2}^S \prod\limits_{i_l=p_l}^{K_l}\left( \frac{r_l(i_l)}{r_{l-1}(i_l)r_l(i_l/p_l)}   \right)^{  \frac{2}{i_l(i_l+1)}\prod\limits_{j<l} \left( 1-\frac{1}{p_j} \right)    }.$$


By taking $S=5,K_1=K_2=1006632960, K_3=50000, K_4=2695, K_5=1000$ we get that 
$$\alpha \leq 1.31843.$$
Together with~\eqref{max_lower}, we get lower and upper estimates within a ratio of
$\frac{1.31843}{1.3183}<1.0001$.

\section{Number of primitive subsets of $[n]$}\label{sec-primitive}
In this section, we highlight how to adapt the algorithm in the previous section to estimate $g(n)^{1/n}$.  

For $g(n)$, define now the lattice $L_l(t)=t\cdot M(n/t)$. Instead of $r'_l(i)$, the relevant parameter now is $R'_l(i)$, which is the number of those antichains of $M_l(i)$ that contain  only numbers that are larger than $i/q$. We then have $g_q(n)\ge (c'_{l,K}+o(1))^n$, where $c'_{l,K}=\prod\limits_{i=1}^K R_l'(i)^{\frac{1}{i(i+1)}\prod\limits_{j=1}^l\left(1-\frac{1}{p_j}\right)}$. 

Denoting by $R_l(i)$ the total number of antichains of $M_l(i)$ and setting $c_{l,K}=\prod\limits_{i=1}^K R_l(i)^{\frac{1}{i(i+1)}\prod\limits_{j=1}^l\left(1-\frac{1}{p_j}\right)}$ and $\eta_{l,K}=1-\sum\limits_{i=1}^K \frac{1}{i(i+1)} \prod\limits_{j=1}^l\left(1-\frac{1}{p_j}\right)|M_l(i)|$ we obtain for every $K,l$ the bounds
$$c'_{l,K}\leq \beta\leq c_{l,K}\cdot 2^{\eta_{l,K}}.$$
For $l=2,K=10^6$ we get $1.55966\leq \beta\leq 1.58852$. So the ratio of lower and upper estimates is $\frac{1.58852}{1.55966}<1.019$.

\subsection{Improved algorithm for numerical estimates for $g(n)^{1/n}$}
The convergence of $g(n)^{1/n}$ can be accelerated in an analogous way as we did in the case of max-size primitive subsets of $[2n]$.

The main difference is that here we do not need to pay special attention on chains $C(u)$. We similarly partition $[n]$ into subsets (posets) $\cC_0,\cC_1,\dots,\cC_k$. If $A\subseteq [n]$ is a primitive set, then for every $i$ the set $A\cap \cC_i$ must also be a primitive set (antichain) in $\cC_i$. Denoting by $g(X)$ the number of max-size primitive subsets of $X$ we obtain that 
\begin{equation}\label{g_upp}
g(n)=g([n])\leq \prod\limits_{i=0}^k g(\cC_i)\leq 2^{|\cC_0|}\prod\limits_{i=1}^k g(\cC_i).
\end{equation}
If we change the partition by merging $\cC_i$ and $\cC_j$ (for some $0<i<j$), then we gain a factor of $\frac{g(\cC_i\cup \cC_j)}{g(\cC_i)g(\cC_j)}\leq 1$.
On the other hand, denoting by $g^*(X)$ the number of those primitive $A$ subsets of $X$ which satisfy that all multiples (in $[n]$) of all elements of $A$ are contained in $X$, we obtain that 
\begin{equation}\label{g_low}
\prod\limits_{i=1}^k g^*(\cC_i)\leq \prod\limits_{i=0}^k g^*(\cC_i)\leq g([n])=g(n).
\end{equation}

\subsection{Numerical estimations for the lower bound}
Let $S,K_0,K_1,\dots,K_S$ be positive integers, for brevity let $\underline{K}=(K_0,K_1,K_2,\dots,K_S)$. We start with taking the $L_S(t)$ lattices with $(t,p_1p_2\dots p_S)=1$ and $n/(K_S+1)<t\leq n$, which gives the lower bound 
$$\lambda_S(\underline{K}):=c'_{S,K_S}=\prod\limits_{i_S=1}^{K_S}\left(R_S'(i_S)\right)^{\frac{1}{i_S(i_S+1)}\prod\limits_{j\leq S}\left(1-\frac{1}{p_{j}}\right)}$$
for $\beta$. Continuing this process, for $S-1\geq l\geq 0$. In $(S-l)$-th step, we will merge (to the current posets) those $L_{l}(t)$ lattices contained in the set of left-over elements  for which $n/K_{l}<t\leq n$ and $(t,p_1p_2\dots p_{l})=1$. The next improvement is the factor
$\lambda_{l}(\underline{K}):=\prod\limits_{i_{l}=1}^{K_{l}}R_{l}'(i_{l})^{w(l,i_{l})},$
where the weight $w(l,i_{l})$ is
$$w(l,i_l):=
\frac{\prod\limits_{j\leq S}\left(1-\frac{1}{p_{j}}\right)}{i_l(i_l+1)}\sum\limits_{\alpha_S=0}^{\infty}\sum\limits_{\alpha_{S-1}=0}^{\infty}\dots \sum\limits_{\alpha_{l+1}=0}^{\infty} \prod\limits_{v=l+1}^{S} \frac{I(i_lp_{l+1}^{\alpha_{l+1}}\dots p_v^{\alpha_v}>K_v)}{p_v^{\alpha_v}}.
$$

Finally, after Step $S$, we obtain the lower bound 
$$\prod\limits_{l=0}^{S}\lambda_l(\underline{K})\leq \beta.$$

Note that assuming $K_0\geq K_1\geq K_2\geq \dots \geq K_S$ the formula for the weight $w(l,i_l)$ simplifies as
\begin{multline*}
w(l,i_l)=
\frac{\prod\limits_{j\leq S}\left(1-\frac{1}{p_{j}}\right)}{i_l(i_l+1)}\sum\limits_{\alpha_S=0}^{\infty}\sum\limits_{\alpha_{S-1}=0}^{\infty}\dots \sum\limits_{\alpha_{l+1}=0}^{\infty}  \frac{I(i_lp_{l+1}^{\alpha_{l+1}}>K_{l+1})}{p_{l+1}^{\alpha_{l+1}}\dots p_S^{\alpha_S}}=\\
=\frac{\prod\limits_{j\leq l+1}\left(1-\frac{1}{p_{j}}\right)}{i_l(i_l+1)} \sum\limits_{\alpha_{l+1}=0}^{\infty}  \frac{I(i_lp_{l+1}^{\alpha_{l+1}}>K_{l+1})}{p_{l+1}^{\alpha_{l+1}}} .
\end{multline*}

Note that the last step here is Step $S$, which gives the improvement by the factor $\lambda_{0}(\underline{K}):=\prod\limits_{i_{0}=1}^{K_{0}}R_{0}'(i_{0})^{w(0,i_{0})}$. Since $R_0'(i_0)$ is the number of those antichains of the 1-element set $\{1\}$ where each element is at least $(i_0+1)/2$, we have $R_0'(1)=2$ and $R_0'(i_0)=1$ for $i_0>1$. That is, $\lambda_{0}(\underline{K})=2^{w(0,1)}$, where the number of left-over elements after Step $(S-1)$. is $(w(0,1)+o(1))n$. This represents that for any left-over element in the interval $(n/2,n]$ we can decide independently whether we would like to add them, or do not add them to the antichain. In the case of max-size independent sets this step would not give any improvement, since the number of max-size antichains in the 1-element set $\{i_0\}$ is 1, even without any restriction on $i_0$.

By taking $S=5, K_1=K_2=2^{20},K_3=960,K_4=196,K_5=98$ we get that 
\begin{equation}\label{gen_lower}
1.571068\leq \beta.
\end{equation}

\subsection{Numerical estimations for the upper bound}

We start with $\cC_0=[n]$, which yield the trivial upper bound $g(n)\leq 2^n$, that is, $\beta\leq 2$.

In Step 1, for every $t\in (n/(K_1+1),n/2]$, in decreasing order, we replace the poset of $t$ and the poset of $2t$ by their union. When we consider $t$, the poset of $t$ is $L_0(t)=\{t\}$, and the poset of $2t$ is $L_1(2t)=\{2t,4t,\dots, 2^{\lfloor\log_2 (n/(2t))\rfloor}t\}$. 
Let $R_l(i)$ be the number of antichains in $M_l(i)$. Let $2\leq i_1\leq K_1$. 
The number of those $t$ for which $n/(i_1+1)<t\leq n/i_1$ is $\frac{n(1+o(1))}{i_1(i_1+1)}$, and the improvement for these values of $t$ is $\frac{R_1(i_1)}{R_0(i_1)R_1(i_1/2)}$, since the chains of $t$, $2t$ and the resulting chain are isomorphic with $M_0(i_1)$ (1-element poset), $M_1(i_1/2)$ and $M_1(i_1)$ respectively. Hence, the upper bound corresponding to the resulting partition is:
$$\beta\leq 2\prod\limits_{i_1=2}^{K_1} \left(\frac{R_1(i_1)}{R_0(i_1)R_1(i_1/2)}\right)^{\frac{1}{i_1(i_1+1)}}.$$
Note that $R_0(i_1)=2$ is the number of antichains of the 1-element set $\{i_1\}$.

For $S\geq 2$, in Step $S$, we choose a positive integer $K_S\leq K_{S-1}$ and for every $t\in (n/(K_S+1),n/p_S]$ satisfying $(t,p_1p_2\dots p_{S-1})=1$, in decreasing order, we merge the poset of $t$ and the poset of $p_St$. When we consider $t$, the poset of $t$ is $L_{S-1}(t)$ and the poset of $p_St$ is $L_S(p_St)$. After Step $S$, we obtain the upper bound
$$\beta\leq 2\prod\limits_{l=1}^S \prod\limits_{i_l=P_l}^{K_l}\left( \frac{R_l(i_l)}{R_{l-1}(i_l)R_l(i_l/p_l)}   \right)^{  \frac{1}{i_l(i_l+1)}\prod\limits_{j\leq l-1} \left( 1-\frac{1}{p_j} \right)    }.$$

By taking $S=5, K_1=K_2=2^{20},K_3=960,K_4=196,K_5=98$ we get 
$$\beta \leq 1.574445.$$
Together with~\eqref{gen_lower}, we get lower and upper estimates within a ratio of
$\frac{1.574445}{1.571068}<1.00215$.

\section{Pairwise coprime}\label{sec-coprime}
In this section, we prove Theorem~\ref{thm-coprime}. For a positive integer $n$, denote by $\Omega(n)$ the number of distinct prime divisors of $n$. Denote by $\pi(n)$ the number of primes at most $n$. Assume that $A\subseteq [n]$ contains pairwise coprime integers, let
	$$A_1=\{a\in A: \Omega(a)\leq 1\},$$
	$$A_2=\{a\in A: \Omega(a)= 2\},$$
	$$A_3=\{a\in A: \Omega(a)\geq 3\},$$
	that is, $A_1,A_2,A_3$ contains the elements having at most 1, exactly 2, at least 3 prime factors, respectively.
	
	The number of choices for $A_1$ is precisely the number of subsets of the set of all primes and $1$, which is at most $2^{\pi(n)+1}$. 
	
	In $A_3$, every element has a prime factor below $n^{1/3}$. As elements in $A_3\subseteq A$ are pairwise coprime, each prime less than $n^{1/3}$ can be a divisor of at most one element in $A_3$. Thus the number of choices for $A_3$ is at most $n^{\pi(n^{1/3})}=e^{o(\sqrt{n})}$.
	
	Let us partition the elements of $A_2$ into two classes: $A_2'$ contains the elements having a prime factor which is at most $\sqrt{n}/\log n$ and $A_2''=A_2\setminus A_2'$ contains the remaining elements. The number of choices for $A_2'$ is at most $n^{\pi(\sqrt{n}/\log n)}=e^{o(\sqrt{n})}$. For each $a\in A_2''$, write $a=pq$ where $\sqrt{n}/\log n<p\le \sqrt{n}\le q<\sqrt{n}\log n$. Similarly, each of the $\pi(\sqrt{n})$ choices of $p$ can only divide at most one $a\in A_2''$; and there are obviously at most $\sqrt{n}\log n$ choices for the corresponding $q=a/p$. Hence, the number of choices for $A_2''$ is at most $(\sqrt{n}\log{n})^{\pi(\sqrt{n})}=e^{(1+o(1))\sqrt{n}}$.
	
	Now, we continue with the lower bound. Let $2=p_{k}<p_{k-1}<\dots<p_1\leq \sqrt{n}\left(1-\frac{1}{\log n}\right)$ be the primes up to $\sqrt{n}\left(1-\frac{1}{\log n}\right)$. We define $A_2$ in the following way: For $i=1,2,\dots,k$ we choose a pair $q_i$ for $p_i$ from the set of primes from the interval $[\sqrt{n},n/p_i]$. The pair of $p_i$ is chosen in such a way that $q_i$ has to be different from the previously chosen $q_1,\dots,q_{i-1}$ primes. Finally, $A_2:=\{p_iq_i: 1\leq i\leq k\}$. The number of choices for $q_i$ is
	\begin{multline*}
	\pi(n/p_i)-\pi(\sqrt{n})-(\pi\left(\sqrt{n}\left(1-\frac{1}{\log n}\right)\right)-\pi(p_i))\gtrsim \\
	\gtrsim \frac{2n/p_i}{\log n}-\frac{2\sqrt{n}}{\log n}-\frac{2\sqrt{n}\left(1-\frac{1}{\log n}\right)}{\log n}+\frac{2p_i}{\log n}\geq \frac{2\sqrt{n}}{(\log n)-1}-\frac{2\sqrt{n}}{\log n}\geq\frac{2\sqrt{n}}{(\log n)^2}
	\end{multline*}
	Since $k\sim \frac{2\sqrt{n}}{\log n}$, the number of choices for $A_2$ is at least $\left(\frac{2\sqrt{n}}{(\log n)^2}\right)^{\frac{(2+o(1))\sqrt{n}}{\log n}}=e^{(1+o(1))\sqrt{n}}.$
	After choosing $A_2$, we can add any subset of the complement of $\{p_1,\dots,p_k,q_1,\dots,q_k\}$, the number of these subsets is $2^{\pi(n)-2\pi(\sqrt{n})}=2^{\pi(n)}e^{o(\sqrt{n})}$. Therefore, the total number of subsets containing pairwise coprime elements is at least $2^{\pi(n)}\cdot e^{(1+o(1))\sqrt{n}}$. This completes the proof of Theorem~\ref{thm-coprime}.

\section{Concluding remarks}\label{sec-conclude}
In this paper, we prove that the limit of $f(n)^{1/n}$ exists and provide an algorithm showing that the limit is about $1.318$. We also determine asymptotically the number of subsets of $[n]$ with pairwise coprime elements. Our algorithm could be useful for other enumeration problems concerning sets with multiplicative constraints.

\medskip

%


\end{document}